\documentclass[12pt]{article}
\usepackage[margin=1in]{geometry}  
\usepackage{amsthm, amssymb}
\usepackage{amsmath}               
\usepackage{amsfonts}              
\usepackage{amsthm}                
\usepackage{verbatim} 
\usepackage{xspace}
\usepackage[pdftex,bookmarks,colorlinks,breaklinks]{hyperref}  
\hypersetup{linkcolor=red,citecolor=blue,filecolor=dullmagenta,urlcolor=darkblue} 
 \usepackage{graphicx}  
  \usepackage[small]{caption} 
 \usepackage{color}	
\usepackage{titlesec}
\titleformat{\section}{\large\bfseries}{\thesection}{.5em}{}
\newtheorem{theorem}{Theorem}[section]

\newtheorem{remark}[theorem]{Remark}

\newcommand{\la}{\lambda}
\newcommand{\e}{\eta}

\begin{document}
\thispagestyle{empty}

\begin{center}
\textbf{\Large Multiply Balanced Edge Colorings of Multigraphs}

\addvspace{\bigskipamount}
Amin Bahmanian, C. A. Rodger\footnote{rodgec1@auburn.edu} \\
 Department of Mathematics and Statistics, 221 Parker Hall\\
 Auburn University, Auburn, AL\\ USA   36849-5310
\end{center}

\typeout{Abstract}
{\footnotesize  
In this paper, a theorem is proved that generalizes several existing amalgamation results in various ways. The main aim is to disentangle a given edge-colored amalgamated graph so that the result is a graph in which the edges are shared out among the vertices in ways that are fair with respect to several notions of balance (such as between pairs of vertices, degrees of vertices in the both graph and in each color class, etc). The connectivity of color classes is also addressed. Most results in the literature on amalgamations focus on the disentangling of amalgamated complete graphs and complete multipartite graphs. Many such results
follow as immediate corollaries to the main result in this paper, which addresses amalgamations of graphs in general, allowing for example the final graph to have multiple edges.  A  new corollary of the main theorem is the settling of the existence of Hamilton decompositions of the family of graphs $K(a_1,\dots, a_p;\lambda_1, \lambda_2)$; such graphs arose naturally in statistical
settings.
 
}

\section{Introduction}

Throughout this paper, all graphs are finite and undirected (possibly with loops and multiple edges). The letters $G$ and $H$ denote graphs. Sets may contain repeated elements (so are really multisets). Each edge is represented by a 2-element multisubset of the vertex set; in particular $\{u,u\}$ represents a loop on the vertex $u$.  A \textit{k-edge-coloring} of $G$ is a mapping ${\cal K}:E(G)\rightarrow C$, where $C$ is a set of $k$ \textit{colors} (often we use $C=\{1,\ldots,k\}$), and the edges of one color form a \textit{color class}.

Informally speaking, \textit{amalgamating} a finite graph $G$ can be thought of as taking $G$, partitioning its vertices, then for each element of the partition squashing the vertices to form a single vertex in the amalgamated graph $H$. Any edge incident with an original vertex in $G$ is then incident with the corresponding new vertex in $H$, and any edge joining two vertices that are squashed together in $G$ becomes a loop on the new vertex in $H$. 

More precisely, $H$ is an \textit{amalgamation} of $G$  if there exists a function $\phi$ called an \textit{amalgamation function} from $V(G)$ onto $V(H)$ and a bijection $\phi':E(G)\rightarrow E(H)$ such that $e$ joining $u$ and $v$ is in $E(G)$ if and only if $\phi'(e)$ joining $\phi(u)$ and $\phi(v)$ is in $E(H)$; We write $\phi(G)=H$. In particular, this requires that $e$ be a loop in $H$ if and only if, in $G$, it either is a loop or joins distinct vertices $u,v$, such that $\phi(u) = \phi(v)$.  (Note that $\phi'$ is completely determined by $\phi$.) Associated with $\phi$ is the \textit {number function}  $\e:V(H)\rightarrow \mathbb N$ defined by $\e(v)=|\phi^{-1}(v)|$, for each $v\in V(H)$. We also shall say that $G$ is a \textit{detachment} of $H$ in which each vertex $v$ of $H$ splits (with respect to $\phi$) into the vertices in $\phi^{-1}(\{v\})$. 

A \textit{detachment} of $H$ is, intuitively speaking, a graph obtained from $H$ by splitting some or all of its vertices into more than one vertex. If $\e$ is a function from $V(H)$ into $\mathbb {N}$, then an \textit{$\e$-detachment} of $H$ is a detachment of $H$ in which each vertex $u$ of $H$ splits into $\e(u)$ vertices. In other words, $G$ is an $\e$-detachment  of $H$ if  there exists an amalgamation function $\phi$ of $G$ onto $H$ such that $|\phi^{-1}(\{u\})|=\e(u)$ for every $u\in V(H)$. Some authors refer to detachments as \textit{disentanglements} (see \cite{LR1, LR2, LR3}).

Since two graphs $G$ and $H$ related in the above manner have an obvious bijection between the edges, an edge-coloring of $G$ or $H$, naturally induces an edge-coloring on the other graph. Hence an amalgamation of a graph with colored edges is a graph with colored edges.

By the \textit {multiplicity} of a pair of vertices  $u,v$ of $G$, we mean the number of edges joining $u$ and $v$ in $G$. A graph is said to be: (i) \textit {almost regular} if there is an integer $d$ such that every vertex has degree $d$ or $d+1$, (ii) \textit {equimultiple} if there is an integer $d$ such that every pair of vertices has multiplicity $d$ or $d+1$, (iii) \textit{$P$-almost-regular} (where $P=\{P_1,\ldots,P_r\}$ is a partition of $V(G)$) if for $1\leq i\leq r$, there is an integer $d_i$ such that each vertex in $P_i$ has degree $d_i$ or $d_i+1$.

In this paper, $K_n$ denotes the complete graph with $n$ vertices, $K_{m,\ldots,m}$ denotes the complete multipartite graph each part having $m$ vertices, and $K(a_1,\ldots, a_p;\lambda_1,\lambda_2)$ denotes a graph with $p$ parts, the $i^{th}$ part having size $a_i$, in which multiplicity of every pair of vertices in the same part (in different parts) is  $\lambda_1$ ($\lambda_2$, respectively). If we replace every edge of $G$ by $\lambda$ multiple edges, then we denote the new graph by $\lambda G$. 

The main goal of this paper is to prove Theorem \ref{mainth}. Informally, it states that for a given $k$-edge-colored graph $H$ and a function $\e:V(H)\rightarrow \mathbb N$, there exists a loopless $\e$-detachment $G$ of $H$ with  amalgamation function $\phi:V(G)\rightarrow V(H)$, $\e$ being the number function associated with $\phi$, such that: 
(i) $G$ and each of its color classes are $P$-almost-regular where $P=\{\phi^{-1}(v):v\in V(H)\}$, (ii) the subgraph of $G$ induced by $\phi^{-1}(v)$ is equimultiple for each $v\in V(H)$, as are each of its color classes, (iii) the bipartite subgraph of $G$ formed by the edges joining vertices in $\phi^{-1}(u)$ to the vertices in $\phi^{-1}(v)$ is equimultiple for every pair of distinct $u,v \in V(H)$, as are each of its color classes, and 
(iv) under certain conditions, the subgraph induced by each color class can be guaranteed to have the same number of components in $G$ as in $H$. 
The conditions (ii) and (iii) can be used to force $G$ to be multigraphs of interest, such as $\lambda K_n, \lambda K_{m,\ldots, m}$, or $K(a_1,\ldots, a_p;\lambda_1,\lambda_2)$.  As in previous results, condition (iv) is especially useful in the context of Hamiltonian decompositions, since it can be used to force connected color classes in $H$ to remain connected in $G$.

A Hamiltonian decomposition of a graph $G$ is a partition of the edges of $G$ into sets, each of which induces a spanning cycle. Hamiltonian decompositions have been studied since 1892, when Walecki \cite{L} proved the classic result that $K_n$ is Hamiltonian decomposable if and only if $n$ is odd. In 1976 Laskar and Auerbach \cite{LA} settled the existence of Hamiltonian decomposition of the complete multipartite graph $K_{m,\ldots,m}$ and of $K_{m,\ldots,m}-F$ where $F$ is a 1-factor. Nash-Williams \cite{NashConj} conjectured that every $2k$-regular graph with at most $4k+1$ vertices has a Hamiltonian decomposition. 

Several techniques have been used for finding Hamiltonian decompositions. The technique of vertex amalgamation, which was developed in the 1980s by Hilton and Rodger \cite{H2, HR}, has proved to be very powerful in constructing Hamiltonian decompositions of various classes of graphs, especially in obtaining embedding results; see also \cite{HJRW,MatJohns,N,RW}. Buchanan \cite{B} used amalgamations to prove that for any $2$-factor $U$ of $K_n$, $n$ odd, $K_n-E(U)$ admits a Hamiltonian decomposition. Rodger and Leach \cite{LR1} solved the corresponding existence problem for complete bipartite graphs, and obtained a solution for complete multipartite graphs when $U$ has no small cycles \cite{LR2}. See also \cite{DBry, McCR} for different approach to solve this problem. Detachments of graphs have also been studied in \cite{BeJacksonJor,JacksonJor}, generalizing some results of Nash-Williams \cite{NashW,NashW85}.

The main theorem of this paper, Theorem \ref{mainth}, not only generalizes several well-known graph amalgamation results, (for example, in \cite{H2, HR, LR1, LR3, RW},Theorem 1, Theorem 1, Theorem 3.1, Theorem 2.1 and Theorem 2.1 respectively all follow as immediate corollaries)), but also provides the right tool to find necessary and sufficient conditions for $K(a_1,\ldots, a_p;\lambda_1,\lambda_2)$ to be Hamiltonian decomposable, as shown in Theorem \ref{hdka1apl1l2}.  The latter graph, $K(a_1,\ldots, a_p;\lambda_1,\lambda_2)$, is of particular interest to statisticians, who consider group divisible designs with two associate classes, beginning over 50 years ago with the work of Bose and Shimamoto \cite{BoSh}. Recently, partitions of the edges of  $K(a_1,\ldots, a_p;\lambda_1,\lambda_2)$ into sets, each of which induces a cycle of length $m$, have been extensively studied for small values of $m$ \cite{FuRod98, FuRod01, FuRodSar}. Theorem \ref{hdka1apl1l2}  provides a companion to this work, settling the problem completely for longest (i.e. Hamiltonian) cycles with a really neat proof.  When $a_1 = \ldots = a_p=a$, we denote $K(a_1,\ldots, a_p;\lambda_1,\lambda_2)$ by $K(a^{(p)};\lambda_1,\lambda_2)$. It is expected that Theorem \ref{mainth} can be used to provide conditions under which one can embed an edge-colored $K(a^{(p)};\lambda_1,\lambda_2)$ into an edge-colored $K(a^{(p+r)};\lambda_1,\lambda_2)$ such that every color class of $K(a^{(p+r)};\lambda_1,\lambda_2)$ induces a Hamiltonian cycle. However obtaining such results will be much more complicated than for companion results for simple graphs, with a complete solution unlikely to be found in the near future.

We describe terminology and notation in Section \ref{term}. Then we prove the main result in Section \ref{mainthsec}. In Section \ref{hdka1ap}, we give necessary and sufficient conditions for the graph $K(a_1,\ldots, a_p;$   $\lambda_1, \lambda_2)$ to be Hamiltonian decomposable as a corollary of the main theorem.

\section{Terminology and more definitions}\label{term}
In this paper, $\mathbb R$ denotes the set of real numbers, $\mathbb{N}$ denotes the set of positive integers, and $\mathbb{Z}_k$ denotes the set of integers $\{1,\ldots,k\}$. If $f$ is a function from a set $X$ into a set $Y$ and $y\in Y$, then $f^{-1}(y)$ denotes the set $\{x\in X:f(x)=y\}$, and $f^{-1}[y]$ denotes $\{x\in X:f(x)=y\}\backslash\{y\}$. 
 If $x, y$ are real numbers, then $\lfloor x \rfloor$ and $\lceil x \rceil$ denote the integers such that $x-1<\lfloor x \rfloor \leq x \leq \lceil x \rceil < x+1$, and $x\approx y$ means $\lfloor y \rfloor \leq x\leq \lceil y \rceil$. We observe that for $x, y, z, x_1,\ldots x_n \in \mathbb{R}, a, b, c \in \mathbb{Z}$, and  $n\in \mathbb{N}$: (i) $a\approx x$ implies $a\in \{\lfloor x \rfloor, \lceil x \rceil \}$, (ii) $x\approx y$ implies $x/n\approx y/n$ (iii) the relation $\approx$ is transitive (but not symmetric), (iv) $x_i\approx x$ for $1\leq i\leq n$ implies $(\sum_{i=1}^n x_i)/n\approx x$, (v) $x\approx y$  and $y<a$ implies $x\leq a$, and (vi) $a=b-c$ and $c\approx x$, implies $a\approx b-x$. These properties of $\approx$ will be used in Section \ref{mainthsec} when required without further explanation.

We make extensive use of edge-colorings in this paper. An edge-coloring of a multigraph is (i) \textit{equalized} if the number of edges colored with any two colors differs by at most one, (ii) \textit{balanced} if for each pair of vertices, among the edges joining the pair, the number of edges of each color differs by at most one from the number of edges of each other color, and (iii) \textit{equitable} if, among the edges incident with each vertex, the number of edges of each color differs by at most one from the number of edges of each other color. Equitable edge-coloring has been used in scheduling and timetabling problems \cite{deW71-2,M}. In \cite{deW71, deW71-2, deW75, deW75-2} de Werra studied balanced equitable edge-coloring of bipartite graphs. His techniques easily allows his coloring to be equalized as well, so we have the following result:
\begin{theorem} \label{BEE}
Every finite bipartite graph has a balanced, equitable and equalized $k$-edge-coloring for each $k\in \mathbb N$. 
\end{theorem}  
If $G$ is a $k$-edge-colored graph, and if $u,v\in V(G)$ and $A, B\subset V(G)$, with $A\cap B=\varnothing$ then  $\ell(u)$ denotes the number of loops incident with vertex $u$, $d(u)$ denotes the degree of vertex $u$ (loops are considered to contribute two to the degree of the incident vertex), and $m(A,B)$ denotes the total number of edges joining vertices in $A$ to vertices in $B$. We refer to  $m(A,B)$ as the \textit{multiplicity} of pair $A, B$, naturally generalizing the multiplicity $m(u,v)$ of a pair of vertices $u,v$ as used in \cite{GrBerge}. In particular by $m(u,A)$ we mean $m(\{u\},A)$. If $G_1,G_2$ are subgraphs of $G$ with $V(G_1)=A$ and  $V(G_2)=B$, then we let $m(G_1,G_2)$ denote $m(A,B)$, and  $m(u, G_1)$ denote $m(\{u\}, A)$.

The subgraph of $G$ induced by the edges colored $j$ is denoted by $G(j)$, and $\omega (G)$ is the number of components of $G$. The neighborhood of vertex $v$, written $N(v)$, denotes the set of all vertices adjacent to $v$ (not including $v$).

\section{Main Theorem} \label{mainthsec}
The main theorem below describes some strong properties that can be guaranteed to be satisfied by some detachment $G$ of a given edge-colored graph $H$. Condition (A1) addresses the issue of $P$-almost-regularity (where $P$ is a partition of $V(G)$), while conditions (A3) and (A5) address the equimultiplicity issue in $G$. Conditions (A1), (A3) and (A5) have  companion conditions (A2), (A4) and (A6),  respectively, that restricts  the graphs considered to the color classes of $G$. Condition (A7) addresses the connectivity issue of each color class of $G$.  
\begin{theorem}  \label {mainth}
Let $H$ be a $k$-edge-colored graph and let $\e$ be a function from $V(H)$ into $\mathbb{N}$ such that for each $w \in V(H)$, $\e (w) = 1$ implies $\ell_H (w) = 0$. Then there exists a loopless $\e$-detachment $G$ of $H$ with amalgamation function $\psi:V(G)\rightarrow V(H)$, $\e$  being the number function associated with $\psi$, such that $G$ satisfies the following conditions:
\begin{itemize}
\item [\textup{(A1)}] $d_G(u) \approx d_H(w)/\e (w)  $ for each $w\in V(H)$ and each $u\in \psi^{-1}(w);$
\item [\textup{(A2)}] $d_{G(j)}(u) \approx d_{H(j)}(w)/\e (w)  $ for each $w\in V(H)$, each $u\in \psi^{-1}(w)$ and each $j\in \mathbb{Z}_k;$
\item [\textup{(A3)}] $m_G(u, u') \approx \ell_H(w)/\binom {\e(w)}{2} $ for each $w\in V(H)$ with $\e (w) \geq 2$ and every pair of distinct vertices $u,u'\in \psi^{-1}(w);$
\item [\textup{(A4)}] $m_{G(j)} (u, u') \approx  \ell_{H(j)}(w)/\binom{\e (w)}{2}  $ for each $w\in V(H)$ with $\e (w) \geq 2$, every pair of distinct vertices $u,u'\in \psi^{-1}(w)$ and each $j\in \mathbb{Z}_k;$
\item [\textup{(A5)}] $m_G(u, v) \approx m_H(w, z)/(\e (w) \e (z)) $ for every pair of distinct vertices $w,z\in V(H)$, each $u\in \psi^{-1}(w)$ and each $v\in \psi^{-1}(z);$
\item [\textup{(A6)}] $m_{G(j)}(u, v) \approx m_{H(j)}(w, z)/(\e (w)\e (z))  $ for every pair of distinct vertices $w,z\in V(H)$, each $u\in \psi^{-1}(w)$, each $v\in \psi^{-1}(z)$ and each $j\in \mathbb{Z}_k;$
\item [\textup{(A7)}] If for some $j\in \mathbb{Z}_k$, $d_{H(j)}(w)/\e (w)$ is an even integer for each $w \in V(H)$, then $\omega(G(j)) = \omega(H(j))$.
\end{itemize}
\end{theorem}
\begin{remark} \textup {
All existing results in \cite{H2, HR, LR1, LR3, RW} study amalgamations for complete graphs or complete multipartite graphs. In these papers, 
Theorem 1, Theorem 1, Theorem 3.1, Theorem 2.1, and Theorem 2.1 respectively are all immediate corollaries of Theorem \ref{mainthsec}. Other results in the literature may have another focus, most notably in \cite{MatJohns, N, RW} where the edge-connectivity of each color class is specified; such results are not generalized by Theorem \ref{mainthsec}.  
}\end{remark}
\begin{proof} Let $H = (V, E)$ and let $n = \sum\limits _{v \in V} (\e (v) -1)$. 
Our proof consists of the following major parts. First we shall describe the construction of a sequence of graphs $H_0=H, H_1,\ldots, H_n$, where $H_{i}$ is an amalgamation of $H_{i+1}$ (so $H_{i+1}$ is a detachment of $H_i$) for $0\leq i\leq n-1$ with amalgamation function $\psi_i$ that combines a vertex with amalgamation number 1 with one other vertex.  
 To construct each $H_{i+1}$ from $H_i$ we will use two bipartite graphs $B_i, B'_i$. Then we will observe some properties of $B'_i$. We will show that  these properties will impose conditions on $H_{i+1}$ in terms of $H_i$. The relations between $H_{i+1}$ and $H_i$ lead to conditions relating each $H_i$, $1\leq i\leq n$ to the initial graph $H$. This will then show that $H_n$ satisfies the conditions (A1)-(A7), so we can let $G=H_n$.  

Initially we let $H_0=H, \e_0=\e$, and we let $\psi_0$ be the identity function from $V$ into  $V$. Now assume that $H_0=(V_0,E_0),\ldots,H_i=(V_i,E_i)$ and $\psi_0,\ldots,\psi_i$ have been defined for some $i\geq 0$. Also assume that $\e_0:V_0\rightarrow\mathbb{N},\ldots, \e_i:V_i\rightarrow\mathbb{N}$ have been defined for some $i\geq 0$ such that  for each $j=0,\ldots,i$ and each $y \in V_j$, $\e_j (y) = 1$ implies $\ell_{H_j} (y) = 0$.  Let $\varphi_i=\psi_0\ldots\psi_i$. If $i=n$, we terminate the construction, letting $G=H_n$ and  $\psi=\varphi_n$. Otherwise, we can select a vertex $y$ of $H_i$ such that $\e_i(y)\geq 2$. $H_{i+1}$ is formed from $H_i$ by detaching a vertex $v_{i+1}$ with amalgamation number 1 from $y$. 

To decide which edge (and loop) to detach from $y$ and to move to $v_{i+1}$, we construct two sequences of bipartite graphs $B_0, \dots, B_{n-1}$ and $B_0', \dots, B_{n-1}'$ together with a sequence $F_0, F_1, \dots, F_{n-1}$ of sets of edges (possibly including loops) with $F_i\subset E(B_i')$ for $i=0,\ldots,n-1$; each edge in $F_i$ corresponds to an edge in $H_i$ which will have one end detached from $y$ and joined to $v_{i+1}$ when forming $H_{i+1}$. 

Let $c_{i1}, \dots, c_{ik}$ and ${\cal L}_i$ be distinct vertices which do not belong to $V_i$. 
Let $B_i$ be a bipartite graph whose vertex bipartition is $\{Q_i, W_i\}$, where  

$$
Q_i = \{c_{i1}, \dots, c_{ik}\} \mbox{ and }  W_i= N_{H_i}(y) \cup \{{\cal L}_i\},
$$
and whose edge set is
$$E(B_i)=\Big(\bigcup\limits_{\scriptstyle \{y,u\}\in E(H_i(j)) \hfill \atop \scriptstyle y\neq u }\{ \{c_{ij},u\}\}\Big) \bigcup  \Big(\bigcup\limits_{\{y,y\}\in E(H_i(j))}\{ \{c_{ij},{\cal L}_i\},\{c_{ij},{\cal L}_i\}\}\Big).$$
Intuitively speaking, for each color $j \in \mathbb{Z}_k$ and each vertex $u \in W_i\backslash \{{\cal L}_i\}$ an edge is placed between $c_{ij}$ and $u$ in $B_i$ for each edge in $H_i(j)$ joining $y$ to $u$. Moreover, two edges are placed between $c_{ij}$ and ${\cal L}_i$ in $B_i$ for each loop incident with $y$ in $H_i(j)$. This is shown in Figure \ref{figure:HiBi}.
\begin{figure}[htbp]
\begin{center}
\scalebox{.80}
{ \includegraphics {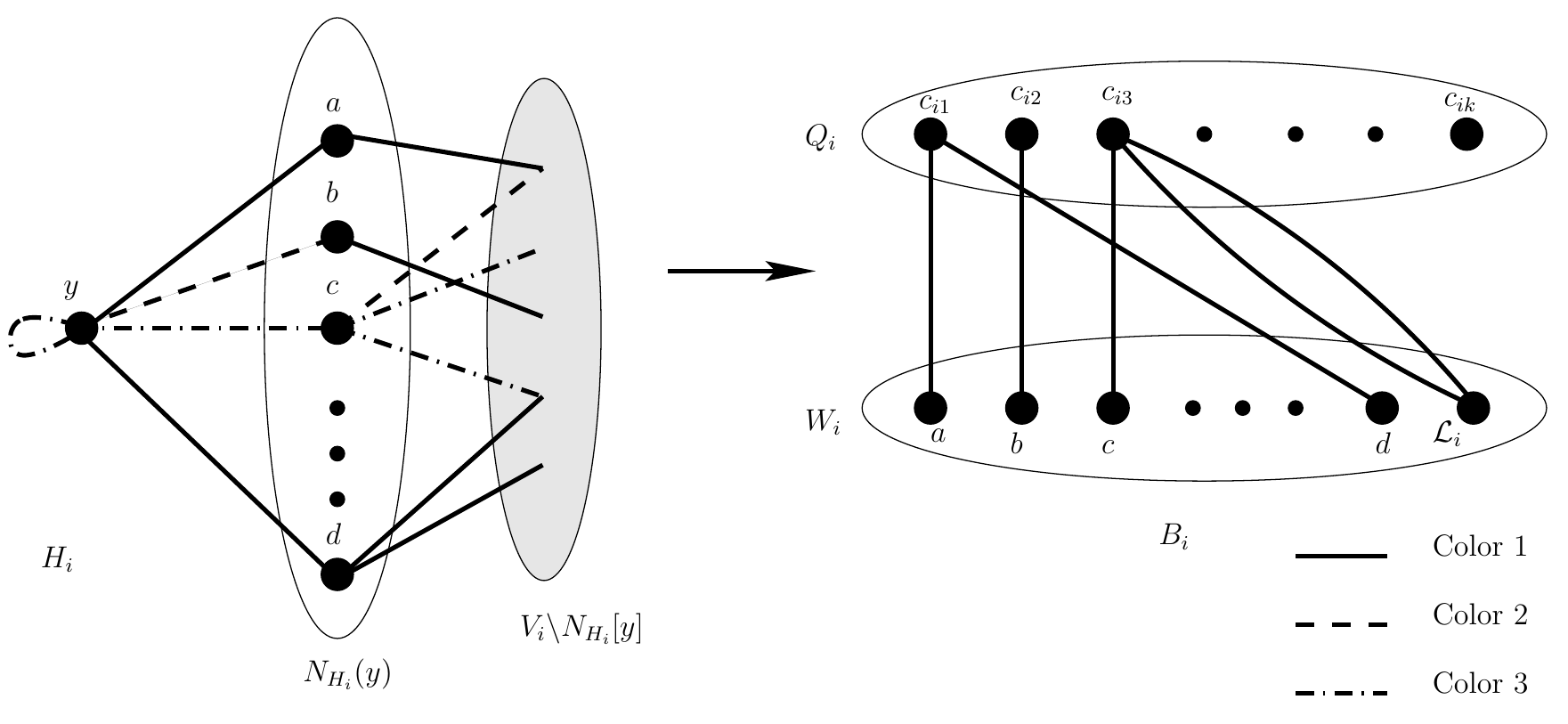} }

\caption{Construction of $B_i$ from $H_i$}
\label{figure:HiBi}
\end{center}
\end{figure}

For $B_i$ we have
\begin{equation}\label{dbiv}
d _{B_i}(v) = \left \{ \begin{array}{ll}
d_{H_i(j)}(y) & \mbox { if } v = c_{ij}   \mbox{ for some } j\in \mathbb{Z}_k\\
2\ell_{H_i}(y) & \mbox { if } v = {\cal L}_i\\ 
m_{H_i}(y,v) & \mbox { otherwise.} \end{array} \right. 
\end{equation}

By Theorem \ref{BEE} we can give $B_i$ an equalized, equitable and balanced $\e _i(y)$-edge-coloring ${\cal K}_i$. Since ${\cal K}_i$ is equitable, for each $1\leq r \leq \e_i(y)$, we have
\begin{equation}\label{dbivr}
d _{B_i(r)}(v) \approx \left \{ \begin{array}{ll}
d_{H_i(j)}(y)/\e_i(y) & \mbox { if } v = c_{ij}   \mbox{ for some } j\in \mathbb{Z}_k\\ 
{2\ell_{H_i}(y)/\e_i(y)} & \mbox { if } v = {\cal L}_i\\ 
{m_{H_i}(y,v)/\e_i(y)} & \mbox { otherwise.}
 \end{array} \right. 
\end{equation}

Now let $T_i$ be formed by a subgraph of $B_i$ induced by the edges colored 1 and 2. Since $\e _i(y) \geq 2$, this is always possible. For each color $j \in \mathbb{Z}_k$ for which 
\begin{equation}\label{concond} \mbox{ for all } v\in V_i, d_{H_i(j)}(v)/\e _i(v) \mbox{ is an even integer, }\end{equation}
 define $\alpha _{ij}  = d_{H_i(j)} (y)/\e_i(y)$. By \eqref{dbivr} for each color class $r$ of ${\cal K}_i$, $d _{B_i(r)}(c_{ij}) \approx d_{H_i(j)}(y)/\e_i(y)$. Therefore since two color classes of ${\cal K}_i$ are chosen to form $T_i$, if \eqref{concond} is satisfied, then  $d_{T_i}(c_{ij})=2 d_{H_i(j)} (y)/\e_i(y)=2\alpha_{ij}$. 

Let $B_i'$ be the bipartite graph whose vertex bipartition is $\{Q_i', W_i\}$, obtained by splitting all the vertices $c_{ij}$ in $T_i$ for each $j\in \mathbb{Z}_k$ for which condition \eqref{concond} holds, into $\alpha _{ij}$ vertices $c_{i,j, 1}, \dots, c_{i,j, \alpha _{ij}}$ all of degree $2$ as described in (M1)-(M2) below. (We don't split vertices $c_{ij}$ in $T_i$ for $j\in \mathbb{Z}_k$ for which condition \eqref{concond} does not hold; but they and their incident edges remain in $B_i'$.)
\begin{itemize}
\item[(M1)] First, as many of $c_{i,j,t}$'s $1\leq t\leq \alpha _{ij}$ as possible are joined by $2$ edges to the same vertex in $W_i$;
\item[(M2)] Then, among all $c_{i,j,t}$'s $1\leq t\leq \alpha _{ij}$  with valency less than 2, as many of them as possible are incident with two edges that correspond to edges in $H_i(j)$ that join $y$ to vertices that are both in the same component of $H_i(j) \backslash \{y\}$.
\end{itemize} 
For each $j\in \mathbb{Z}_k$ that satisfies condition \eqref{concond}, we let ${\cal C}_{ij}=\bigcup\limits_{t=1}^{\alpha_{ij}}{\{c_{i,j,t}\}}$. Otherwise, we let  ${\cal C}_{ij}=\{c_{ij}\}$. 
By Theorem \ref{BEE}, we can give $B_i'$ an equalized, equitable and balanced $2$-edge-coloring ${\cal K}_i'$.  This gives us two color classes either of which can be chosen to be $F_i$, say the edges colored $1$ are chosen. Since ${\cal K}_i'$ is equitable, we have
\begin{equation}\label{dbpiv1}
d _{B_i'(1)}(v)  \approx \left \{ \begin{array}{ll}
d_{H_i(j)}(y)/\e_i(y) & \mbox { if } v = c_{ij}   \mbox{ for } j\in \mathbb{Z}_k \mbox{ for which }\eqref{concond} \mbox{ does not hold }\\ 
 1 & \mbox { if } v \in {\cal C}_{ij}\mbox{ for } j\in \mathbb{Z}_k \mbox{ for which }\eqref{concond} \mbox{  holds }\\ 
 {2\ell_{H_i}(y)/\e_i(y)} & \mbox { if } v = {\cal L}_i\\ 
 {m_{H_i}(y,v)/\e_i(y)} & \mbox { otherwise.}
 \end{array} \right. 
\end{equation}
Now we let 
$$A_{ij}=\Big(\bigcup\limits_{\scriptstyle\{c,v\}\in F_i \atop   {\scriptstyle c\in{\cal C}_{ij}  \atop   \hfill}}  {\{y,v\} } \Big)\bigcup \Big(\bigcup\limits_{\scriptstyle\{c,{\cal L}_i\}\in F_i \atop   {\scriptstyle c\in {\cal C}_{ij} \atop   \hfill}}  {\{y,y\} } \Big)$$
and 
$$B_{ij}= \Big(\bigcup\limits_{\scriptstyle\{c,v\}\in F_i \atop   {\scriptstyle c\in{\cal C}_{ij}  \atop   \hfill}}  {\{v_{i+1},v\} } \Big )\bigcup \Big(\bigcup\limits_{\scriptstyle\{c,{\cal L}_i\}\in F_i \atop   {\scriptstyle c\in{\cal C}_{ij}\atop   \hfill}}  {\{v_{i+1},y\} } \Big ),$$
where $v_{i+1}$ is a vertex which does not belong to $V_i$. Let $V_{i+1}=V_i\cup\{v_{i+1}\}$, and let $\psi _{i+1}$ be a function from $V_{i+1}$ onto $V_i$ such that 
$$
\psi _{i+1} (v) = \left \{ \begin{array}{ll}
y & \mbox { if } v = v _{i+1}  \\
v & \mbox { otherwise.} 
\end{array} \right. 
$$
Let $H_{i+1}=(V_{i+1},E_{i+1})$ be the $\psi _{i+1}$-detachment of $H_i$ such that for each $j\in \mathbb Z_k$ $$E(H_{i+1}(j))=(E(H_i(j))\backslash A_{ij}) \cup B_{ij},$$ and $E_{i+1}=\bigcup_{j=1}^{k}E(H_{i+1}(j))$.

Intuitively speaking, $H_{i+1}$ is formed as follows. Each edge $\{c, v\}\in F_i$ with $c\in{\cal C}_{ij} $ and $v \in W_i\backslash \{{\cal L}_i\}$ directly corresponds to an edge $\{y, v\}$ in $H_i(j)$;  replace $\{y,v\}$ with the edge $\{v, v_{i+1}\}$ colored $j$ in $H_{i+1}$. So in forming $H_{i+1}(j)$ from $H_i(j)$ the end of this edge is detached from $v$ and joined to the new vertex $v_{i+1}$ instead. Moreover, we remove $m_{B_i'(1)}({\cal C}_{ij},{\cal L}_i)$ loops colored $j$ incident with $y$ in $H_i$ and we replace them with $m_{B_i'(1)}({\cal C}_{ij},{\cal L}_i)$ edges colored $j$ joining $y$ to $v_{i+1}$ in $H_{i+1}$. Note that since ${\cal K}_i'$ is balanced, $\e _i (y) \geq 2$ and $\lceil d_{B_i'} ({\cal L}_i)/2 \rceil \leq \lceil d_{B_i} ({\cal L}_i)/2 \rceil = \ell_{H_i}(y)$, at most half of the edges in $B_i'$ incident with ${\cal L}_i$ are colored $1$, so there are indeed $m_{B_i'(1)}({\cal C}_{ij},{\cal L}_i)$ loops incident with $y$ in $H_i$ (recall that each loop in $H_i$ corresponds to two edges in $B_i'$).

Obviously, $\psi_{i+1}$ is an amalgamation function from $H_{i+1}$ into $H_i$. Let $\e _{i+1}$ be the function from $V_{i+1}$ into $\mathbb N$ such that
$$
\e _{i+1}(v) = \left \{ \begin{array}{ll}
1 & \mbox { if } v = v_{i+1}\\
\e _i(v) -1 & \mbox { if } v = y\\ 
\e _i(v) & \mbox { otherwise.} \end{array} \right. \\ 
$$

We now check that $B_i'$, described above, satisfies the following conditions for each color $j\in \mathbb{Z}_k:$
\begin{itemize}
\item [\textup{(P1)}] $m_{B_i'(1)}({\cal C}_{ij},{\cal L}_i)\approx 2\ell_{H_i(j)}(y)/\e_i(y);$ 
\item [\textup{(P2)}] $m_{B_i'(1)}({\cal C}_{ij},v)\approx  m_{H_i(j)}(y,v)/\e_i(y)$ for each $v\in W_i\backslash\{{\cal L}_i\};$
\item [\textup{(P3)}] $m_{B_i'(1)}(Q_i',W_i)\approx   d_{H_i}(y)/\e_i(y);$
\item [\textup{(P4)}] $m_{B_i'(1)}({\cal C}_{ij},W_i)\approx  d_{H_i(j)}(y)/\e_i(y).$
\end{itemize}
 In order to prove (P1) and (P2) first we show that 
 $$m_{B_i'(1)}({\cal C}_{ij},v)\approx \frac{m_{B'_i}({\cal C}_{ij},v)}{2}  \mbox{  for each } v\in W'_i .$$ 
 There are two cases:
\begin{itemize}
\item Case 1: ${\cal C}_{ij}=\{c_{ij}\}$. Since ${\cal K}_i'$ is balanced, $$m_{B_i'(1)}({\cal C}_{ij},v)=m_{B_i'(1)}(c_{ij},v)\approx \frac{m_{B'_i}(c_{ij},v)}{2}=\frac{m_{B_i'}({\cal C}_{ij},v)}{2}.$$ 
\item Case 2: ${\cal C}_{ij}=\bigcup\limits_{t=1}^{\alpha_{ij}}{\{c_{i,j,t}\}}$. By (M1), among all vertices in ${\cal C}_{ij}$, there are exactly $\lfloor m_{B_i'}({\cal C}_{ij},v)/2 \rfloor$ vertices of degree 2 which are joined to $v$ (at most one vertex in ${\cal C}_{ij}$ is joined to $v$ by one edge). Since ${\cal K}_i'$ is balanced(or equitable), among these vertices of degree 2, exactly one of them is joined to $v$ by an edge colored 1. Therefore
$$m_{B_i'(1)}({\cal C}_{ij},v)=\sum\limits_{t = 1}^{\alpha _{ij} } m_{B_i'(1)}(c_{i,j,t},v)\approx\frac{m_{B_i'}({\cal C}_{ij},v)}{2}.$$
\end{itemize} 
Clearly $m_{B_i'}({\cal C}_{ij},v)=m_{T_i}(c_{ij},v)=m_{B_i(1)}(c_{ij},v)+m_{B_i(2)}(c_{ij},v)$. If $v={\cal L}_i$, from the definition of $B_i$ it follows that $m_{B_i}(c_{ij},{\cal L}_i)=2\ell_{H_i(j)}(y)$. Since ${\cal K}_i$ is balanced, for each $1\leq r \leq \e_i(y)$ we have $m_{B_i(r)}(c_{ij},{\cal L}_i)\approx 2\ell_{H_i(j)}(y)/\e_i(y)$. Therefore 
$$m_{B_i'(1)}({\cal C}_{ij},{\cal L}_i)  \approx  \frac{m_{B_i'}({\cal C}_{ij},{\cal L}_i)}{2}= \frac{m_{T_i}(c_{ij},{\cal L}_i)}{2}\approx  \frac{2\ell_{H_i(j)}(y)}{\e_i(y)}.$$
This proves (P1).

Now let $v\in W_i\backslash\{{\cal L}_i\}$.  From the definition of $B_i$ it follows that $m_{B_i}(c_{ij},v)=m_{H_i(j)}(y,v)$. Since ${\cal K}_i$ is balanced, for each $1\leq r \leq \e_i(y)$ we have $m_{B_i(r)}(c_{ij},v)\approx m_{H_i(j)}(y,v)/\e_i(y)$. Therefore 
$$m_{B_i'(1)}({\cal C}_{ij},v)  \approx  \frac{m_{B_i'}({\cal C}_{ij},v)}{2}= \frac{m_{T_i}(c_{ij},v)}{2}\approx  \frac{m_{H_i(j)}(y,v)}{\e_i(y)}.$$
This proves (P2). 

Since ${\cal K}_i'$ is equalized, $m_{B_i'(1)}(Q_i',W_i)=|E(B_i'(1))|\approx m_{B'_i}(Q_i',W_i)/2$. Clearly $m_{B_i'}(Q_i',W_i)=|E(B_i')|=m_{T_i}(Q_i,W_i)=m_{B_i(1)}(Q_i,W_i)+m_{B_i(2)}(Q_i,W_i)$. From the definition of $B_i$ it follows that $m_{B_i}(Q_i,W_i)=|E(B_i)|=d_{H_i}(y)$. Since ${\cal K}_i$ is equalized, for each $1\leq r \leq \e_i(y)$ we have $m_{B_i(r)}(Q_i,W_i)=|E(B_i(r))|\approx d_{H_i}(y)/\e_i(y)$. Therefore 
$$m_{B_i'(1)}(Q_i',W_i)  \approx  \frac{m_{B'_i}(Q_i',W_i)}{2}= \frac{m_{T_i}(Q_i,W_i)}{2}\approx \frac{d_{H_i}(y)}{\e_i(y)}.$$
This proves (P3).
 
In order to prove (P4), there are two cases:
\begin{itemize}
\item Case 1: ${\cal C}_{ij}=\{c_{ij}\}$. From \eqref{dbpiv1} it follows that 
$$m_{B_i'(1)}({\cal C}_{ij},W_i)  = m_{B_i'(1)}(c_{ij},W_i)= d_{B'_i(1)}(c_{ij})\approx \frac{d_{H_i(j)}(y)} {\e_i(y)}.$$
\item Case 2: ${\cal C}_{ij}=\bigcup\limits_{t=1}^{\alpha_{ij}}{\{c_{i,j,t}\}}$. In this case $m_{B_i'(1)}({\cal C}_{ij},W_i) = \sum\limits_{t = 1}^{\alpha _{ij} } m_{B_i'(1)}(c_{i,j,t},W_i)$. From \eqref{dbpiv1} it follows that
$$m_{B_i'(1)}({\cal C}_{ij},W_i)  =  \sum\limits_{t = 1}^{\alpha _{ij} }1= \alpha_{ij}=\frac{d_{H_i(j)}(y)} {\e_i(y)}.$$
\end{itemize}
This proves (P4).


Most of the conditions that $H_{i+1}$ must satisfy, are numerical, and we consider them first. The reader who is more interested in the connectivity issue, namely property (A7), may wish to jump to the consideration of conditions (D1)-(D2) on the last two pages of this section.  

Using \eqref{dbpiv1} and (P1)-(P4), now we show that $H_{i+1}$, described above, satisfies the following conditions:
\begin{itemize}

\item [\textup{(B1)}]  $\ell _{H_{i+1}}(y) \approx  \ell _{H_i}(y)(\e_{i+1} (y) - 1)/\e _i(y);$
\item [\textup{(B2)}]  $\ell _{H_{i+1}(j)}(y) \approx   \ell _{H_i(j)}(y)(\e_{i+1} (y) - 1)/\e _i(y)$ for each $j\in \mathbb{Z}_k;$
\item [\textup{(B3)}]
\begin{itemize}
\item  [\textup{(i)}] $d_{H_{i+1}}(y)/\e_ {i+1}(y) \approx   d_{H_i}(y)/\e _i(y),$ 
\item  [\textup{(ii)}] $d_{H_{i+1}}(v_{i+1}) \approx  d_{H_i}(y)/\e_i(y);$
\end{itemize}
\item [\textup{(B4)}]For each $j\in \mathbb{Z}_k$ 
\begin{itemize}
\item  [\textup{(i)}] $d_{H_{i+1}(j)}(y)/ \e_ {i+1}(y) \approx  d_{H_i(j)}(y)/\e _i(y),$ 
\item  [\textup{(ii)}] $d_{H_{i+1}(j)}(v_{i+1}) \approx  d_{H_i(j)}(y)/\e_i(y);$
\end{itemize}
\item [\textup{(B5)}] For each $v\in N_{H_i}(y)$
\begin{itemize}
\item  [\textup{(i)}]  $m_{H_{i+1}} (y, v)/ \e_ {i+1}(y) \approx m_{H_i} (y, v) /\e _i(y),$
\item  [\textup{(ii)}] $m_{H_{i+1}} (v_{i+1}, v) \approx  m_{H_i} (y, v) /\e _i(y),$
\item  [\textup{(iii)}] $m_{H_{i+1}}(y, v_{i+1})/\e_{i+1}(y) \approx \ell_{H_i}(y)/\binom{\e _i(y)}{2};$
\end{itemize}
\item [\textup{(B6)}] For each $v\in N_{H_i}(y)$, and each $j\in \mathbb{Z}_k$
\begin{itemize}
\item  [\textup{(i)}]  $m_{H_{i+1}(j)} (y, v) /\e_ {i+1}(y) \approx  m_{H_i}(j) (y, v)/\e _i(y),$
\item  [\textup{(ii)}] $m_{H_{i+1}(j)} (v_{i+1}, v) \approx   m_{H_{i(j)}} (y, v) /\e _i(y),$
\item  [\textup{(iii)}] $m_{H_{i+1}(j)}(y, v_{i+1})/\e_{i+1}(y) \approx  \ell_{H_i(j)}(y)/\binom{\e _i(y)}{2}.$
\end{itemize} 
\end{itemize}
Note that $\e_{i+1}(y)=\e_i(y)-1$. Let us fix $v\in N_{H_i}(y)$, and $j\in \mathbb{Z}_k$. 

From the construction of $H_{i+1}$, we have $ \ell _{H_{i+1}}(y)= \ell _{H_{i}}(y)-d_{B_i'(1)}({\cal L}_i)$. By \eqref{dbpiv1}, $d_{B_i'(1)}({\cal L}_i)\approx  2\ell_{H_i}(y)/\e_i(y)$. Hence 

 $$ \ell_{H_{i+1}}(y)  \approx   \ell _{H_{i}}(y)-\frac{2\ell_{H_i}(y)}{\e_i(y)}= \frac{ \ell _{H_i}(y)(\e_{i} (y) - 2)}{\e _i(y)} = \frac{ \ell _{H_i}(y)(\e_{i+1} (y) - 1)}{\e _i(y)}.$$

 This completes the proof of (B1).
 
Clearly, $\ell _{H_{i+1}(j)}(y)= \ell _{H_{i}(j)}(y)-m_{B_i'(1)}({\cal C}_{ij},{\cal L}_i)$. By (P1), $m_{B_i'(1)}({\cal C}_{ij},{\cal L}_i)\approx 2\ell_{H_i(j)}(y)/\e_i(y)$. Hence 

$$ \ell _{H_{i+1}(j)}(y) \approx  \ell _{H_{i}(j)}(y)-\frac{2\ell_{H_i(j)}(y)}{\e_i(y)}= \frac {\ell _{H_i(j)}(y)(\e_{i} (y) - 2)}{\e _i(y)} =  \frac {\ell _{H_i(j)}(y)(\e_{i+1} (y) - 1)}{\e _i(y)} .$$
This completes the proof of (B2).

Construction of $H_{i+1}$ follows that, $d_{H_{i+1}}(y)=d_{H_i}(y)-m_{B_i'(1)}(Q_i',W_i)$, and $d_{H_{i+1}}(v_{i+1})=m_{B_i'(1)}(Q_i',W_i)$. By (P3), $m_{B_i'(1)}(Q_i',W_i)\approx d_{H_i}(y)/\e_i(y)$. Hence

$$d_{H_{i+1}}(y)  \approx  d_{H_i}(y)-\frac{d_{H_i}(y)} {\e_i(y)}= \frac{d_{H_i}(y)(\e _i(y)-1)} {\e_i(y)}=  \frac{d_{H_i}(y)\e _{i+1}(y)} {\e_i(y)},$$
 and $d_{H_{i+1}}(v_{i+1}) \approx d_{H_i}(y)/\e_i(y)$. 
This completes the proof of (B3).

From the construction of $H_{i+1}$, we have that $d_{H_{i+1}(j)}(y)=d_{H_i(j)}(y)-m_{B_i'(1)}({\cal C}_{ij},W_i)$, and $d_{H_{i+1}(j)}(v_{i+1})=m_{B_i'(1)}({\cal C}_{ij},W_i)$. By (P4), $m_{B_i'(1)}({\cal C}_{ij},W_i) \approx  d_{H_i(j)}(y)/\e_i(y)$. Hence 

$$d_{H_{i+1}(j)}(y) \approx d_{H_i(j)}(y)- \frac{d_{H_i(j)}(y)} {\e_i(y)}= \frac{d_{H_i(j)}(y)(\e _i(y)-1)} {\e_i(y)}= \frac{d_{H_i(j)}(y)\e _{i+1}(y)} {\e_i(y)},$$
and $d_{H_{i+1}(j)}(v_{i+1}) \approx  d_{H_i(j)}(y)/\e_i(y)$. This completes the proof of (B4).

It is easy to see that, $m_{H_{i+1}}(y,v)=m_{H_{i}}(y,v)-m_{B_i'(1)}(Q_i',v)=m_{H_{i}}(y,v)-d_{B_i'(1)}(v)$, and $m_{H_{i+1}}(v_{i+1},v)=d_{B_i'(1)}(v)$. By \eqref{dbpiv1}, $d_{B_i'(1)}(v) \approx m_{H_i}(y,v)/\e_i(y)$. Hence 

$$m_{H_{i+1}}(y,v)  \approx  m_{H_{i}}(y,v)-\frac{m_{H_i}(y,v)} {\e_i(y)}= \frac{m_{H_i}(y,v)(\e_i(y)-1)} {\e_i(y)}= \frac{m_{H_i}(y,v)\e_{i+1}(y)} {\e_i(y)},$$ 
and $m_{H_{i+1}}(v_{i+1},v) \approx m_{H_i}(y,v)/\e_i(y)$. Moreover, $m_{H_{i+1}}(y,v_{i+1})=m_{B_i'(1)}(Q_i',{\cal L}_i)=d_{B_i'(1)}({\cal L}_i)$. By \eqref{dbpiv1}, $d_{B_i'(1)}({\cal L}_i) \approx 2\ell_{H_i}(y)/\e_i(y)$. Therefore $m_{H_{i+1}}(y,v_{i+1}) \approx 2\ell_{H_i}(y)/\e_i(y)$. Hence

$$\frac{m_{H_{i+1}}(y, v_{i+1})}{\e_{i+1}(y)}  \approx   \frac{2\ell_{H_i}(y)}{\e_i(y)\e_{i+1}(y)}= \frac {\ell_{H_i}(y)}{\binom{\e _i(y)}{2}}.$$
This completes the proof of (B5).

Finally, from the construction of $H_{i+1}$, $m_{H_{i+1}(j)}(y,v)=m_{H_{i}(j)}(y,v)-m_{B_i'(1)}({\cal C}_{ij},v)$, and $m_{H_{i+1}(j)}(v_{i+1},v)=m_{B_i'(1)}({\cal C}_{ij},v)$. By (P2), $m_{B_i'(1)}({\cal C}_{ij},v) \approx m_{H_i(j)}(y,v)/\e_i(y)$. Hence 

$$m_{H_{i+1}}(y,v) \approx m_{H_{i}(j)}(y,v)-\frac{m_{H_i(j)}(y,v)} {\e_i(y)}= \frac{m_{H_i(j)}(y,v)(\e_i(y)-1)} {\e_i(y)}= \frac{m_{H_i(j)}(y,v)\e_{i+1}(y)} {\e_i(y)},$$
 and $m_{H_{i+1}(j)}(v_{i+1},v) \approx m_{H_i(j)}(y,v)/\e_i(y)$. Moreover, $m_{H_{i+1}(j)}(y,v_{i+1})=m_{B_i'(1)}({\cal C}_{ij},{\cal L}_i)$. By (P1), $m_{B_i'(1)}({\cal C}_{ij},{\cal L}_i) \approx 2\ell_{H_i(j)}(y)/\e_i(y)$. Therefore  $m_{H_{i+1}(j)}(y,v_{i+1}) \approx 2\ell_{H_i(j)}(y)/\e_i(y)$. Hence 

$$\frac{m_{H_{i+1}(j)}(y, v_{i+1})}{\e_{i+1}(y)}  \approx  \frac{2\ell_{H_i(j)}(y)}{\e_i(y)\e_{i+1}(y)}= \frac {\ell_{H_i(j)}(y)}{\binom{\e _i(y)}{2}}.$$
This completes the proof of (B6).
\\

Recall that $\varphi_i=\psi_0\ldots\psi_i$, that $\psi_{0}:V\rightarrow V$, and that $\psi_{i}:V_{i}\rightarrow V_{i-1}$ for $i> 0$. Therefore $\varphi_i: V_i\rightarrow V$ and thus $\varphi_i^{-1}: V\rightarrow V_i$. 
Now we use (B1)-(B6) to prove that for $0 \leq i \leq n$, $H_i$ satisfies the following conditions:
\begin{itemize}
\item [\textup{(C1)}]  
\begin{itemize}
\item  [\textup{(i)}] $\ell _{H_i}(w)/\binom {\e _i(w)}{2}  \approx  \ell_H(w)/\binom {\e(w)}{2} $ for each $w\in V$ with $\e (w) \geq 2$, $\e _i(w) \geq 2$,
\item  [\textup{(ii)}] $\ell _{H_i}(w) =\ell _{H_i} (v_r) = 0$ for each $w\in V$ with $\e_i (w) =1$ and each $1\leq r\leq i;$
\end{itemize}
\item [\textup{(C2)}] $\ell _{H_i(j)}(w)/\binom {\e _i(w)}{2}\approx \ell_{H(j)}(w) /\binom {\e (w)}{2}$ for each $w\in V$ with $\e (w) \geq 2$, $\e _i(w) \geq 2$ and each $j\in \mathbb{Z}_k;$
\item [\textup{(C3)}] For each $w\in V$
\begin{itemize}
\item  [\textup{(i)}]  $d_{H_i}(w)/ \e _i(w) \approx d_H(w)/\e (w),$
\item  [\textup{(ii)}]  $d_{H_i}  (v_r) \approx d_H(w)/\e (w) $ for each $v_r \in \varphi_i ^{-1}[w];$
\end{itemize}
\item [\textup{(C4)}] For each $w\in V$ and each $j\in \mathbb{Z}_k$
\begin{itemize}
\item  [\textup{(i)}]  $d_{H_i(j)} (w) /\e _i(w) \approx d_{H(j)}(w)/\e (w),$
\item  [\textup{(ii)}] $d_{H_i(j)} (v_r) \approx d_{H(j)}(w)/\e (w)$ for each $v_r \in \varphi_i ^{-1}[w];$
\end{itemize}
\item [\textup{(C5)}] For each $w\in V$
\begin{itemize}
\item  [\textup{(i)}]  $m_{H_i}   (w, v_r)/\e _i(w) \approx \ell_H(w)/\binom {\e (w)}{2}$ for each $v_r \in \varphi_i ^{-1}[w],$
\item  [\textup{(ii)}]  $m_{H_i} (v_r, v_s) \approx \ell_H(w)/\binom {\e (w)}{2}$ for every pair of distinct vertices $v_r,v_s \in \varphi_i ^{-1}[w];$
\end{itemize}
\item [\textup{(C6)}]  For each $w\in V$, and each $j\in \mathbb{Z}_k$
\begin{itemize}
\item  [\textup{(i)}]  $m_{H_i(j) }(w, v_r)/\e _i(w) \approx \ell _{H(j)}(w)/\binom {\e (w)}{2}$ for each $v_r \in \varphi_i ^{-1}[w],$
\item  [\textup{(ii)}]  $m_{H_i(j)}(v_r, v_s) \approx \ell _{H(j)}(w)/\binom {\e (w)}{2}$ for every pair of distinct vertices $v_r,v_s \in \varphi_i ^{-1}[w];$ 
\end{itemize}
\item [\textup{(C7)}] For every pair of distinct vertices $w,z\in V$ 
\begin{itemize}
\item  [\textup{(i)}] $m_{H_i} (w, z)/ (\e _i(w) \e _i(z)) \approx  m_H(w, z)/(\e (w) \e(z)),$
\item  [\textup{(ii)}] $m_{H_i}(v_r, v_s) \approx m_H(w, z)/(\e (w) \e (z)) $ for each $v_r\in\varphi_{i}^{-1}[w]$  and each $v_s\in\varphi_{i}^{-1}[z],$
\item  [\textup{(iii)}] $m_{H_i} (w, v_s)/\e _i(w)  \approx m_H(w,    z)/(\e (w) \e (z))$ for each $v_s\in \varphi^{-1}_i[z];$
\end{itemize}
\item [\textup{(C8)}] For every pair of distinct vertices $w,z\in V$, and each $j\in \mathbb{Z}_k$ 
\begin{itemize}
\item  [\textup{(i)}] $m_{H_i(j)} (w, z)/( \e _i(w) \e _i(z)) \approx m_{H(j)}(w, z)/(\e (w) \e(z)),$
\item  [\textup{(ii)}] $m_{H_i(j)}(v_r, v_s) \approx m_{H(j)}(w, z)/(\e (w) \e (z)) $ for each $v_r\in\varphi_{i}^{-1}[w]$  and each $v_s\in\varphi_{i}^{-1}[z],$
\item  [\textup{(iii)}] $m_{H_i(j)} (w, v_s) /\e _i(w) \approx m_{H(j)}(w,    z)/(\e (w) \e (z)) $ for each $v_s\in \varphi^{-1}_i[z].$
\end{itemize}
\end{itemize}
Let $w,z$ be an arbitrary pair of distinct vertices of $V$, and let $j\in \mathbb{Z}_k$. We prove (C1)-(C8) by induction. Let us first verify (C1)-(C8) for $i=0$. Recall that $H_0=H$, and $\e _0(w)=\e(w)$.

If $\e(w)\geq 2$, obviously $\ell _{H_0}(w)/\binom {\e _0(w)}{2}  =  \ell_H(w)/\binom {\e(w)}{2} $. If $\e(w)=1$, by hypothesis of Theorem \ref{mainth}, $\ell _{H} (w) = 0$. This proves (C1) for $i=0$. (C2) can be proved in a similar way.  Obviously $d_{H_0}(w)/ \e _0(w) = d_H(w)/\e (w)$ and (C3)(ii) is obvious, so this proves (C3) for $i=0$. The proof for (C4) is similar and (C5)-(C8) are sufficiently obvious.

Now we will show that if $H_i$ satisfies the conditions (C1) - (C8) for some $i< n$, then $H_{i+1}$ (formed from $H_i$ by detaching $v_{i+1}$ from the vertex $y$) satisfies these conditions by replacing $i$ with $i+1$; we denote the corresponding conditions for $H_{i+1}$ by (C1)$'$-(C8)$'$. If $\e_{i+1}(w)=\e_i(w)$, then (C1)$'$-(C6)$'$ are obviously true. So we just check (C1)$'$-(C6)$'$ in the case where $w=y$.  Also if $\e_{i+1}(w)=\e_i(w)$ and $\e_{i+1}(z)=\e_i(z)$, then (C7)$'$-(C8)$'$ are  clearly true. So in order to prove  
(C7)$'$ - (C8)$'$ we shall assume that either $\e_{i+1}(w)=\e_i(w)-1$ or $\e_{i+1}(z)=\e_i(z)-1$. (so $y\in \{w,z\}$; the asymmetry in condition (iii) of (C7)$'$ and (C8)$'$ prevents us from assuming that $w=y$.) 
\begin{itemize}
\item [\textup{(C1)$'$}]  If $\e_{i+1}(y)\geq 2$, by (B1) $\ell _{H_{i+1}}(y) \approx  \ell _{H_i}(y)(\e_{i+1} (y) - 1)/\e _i(y)$, and by (C1)(i) of the induction hypothesis, $\ell _{H_i}(y)/\binom {\e _i(y)}{2}  \approx  \ell_H(y)/\binom {\e(y)}{2}$. Also note that $\binom{\e_i(y)}{2}=\e_i(y)(\e_{i}(y)-1)/2$. Therefore 
\begin{eqnarray*}
\frac{\ell _{H_{i+1}}(y)}{\binom {\e _{i+1}(y)}{2}}& \approx & \frac {\ell_{H_i}(y)(\e_{i+1} (y) - 1)}{\binom {\e _{i+1}(y)}{2}\e_i(y)} =\frac{\ell _{H_i}(y)}{\binom {\e _i(y)}{2}}  \approx \frac{ \ell_H(y)}{\binom {\e(y)}{2}}.
 \end{eqnarray*}
This proves (C1)$'$(i). 

Clearly $ \ell _{H_{i+1}} (v_{i+1}) = 0$ and $ \ell _{H_{i+1}} (v_r) = \ell _{H_i} (v_r)=0$ for each $1\leq r\leq i$. Therefore $ \ell _{H_{i+1}} (v_r) = 0$ for each $1\leq r\leq i+1$. Also if $\e_{i+1}(y)= 1$, by (B1) $\ell _{H_{i+1}}(y) =0$. This proves (C1)$'$(ii).
\item [\textup{(C2)$'$}] The proof is similar to the proof of (C1)$'$(i), following from (B2) and (C2) of the induction hypothesis.  
\item [\textup{(C3)$'$}] By (B3)(i), $d_{H_{i+1}}(y)/\e_ {i+1}(y) \approx  d_{H_i}(y) /\e _i(y)$, and by (C3)(i) of the induction hypothesis, $d_{H_i}(y)/ \e _i(y) \approx d_H(y)/\e (y)$. Therefore 
\begin{eqnarray*}
\frac{d_{H_{i+1}}(y)}{\e_ {i+1}(y)}& \approx & \frac{d_H(y)}{\e (y)}.
 \end{eqnarray*}
This proves (C3)$'$(i). 

By (B3)(ii), $d_{H_{i+1}}(v_{i+1}) \approx  d_{H_i}(y)/\e_i(y)$, and by (C3)(ii) of the induction hypothesis, $d_{H_i}  (v_r) \approx d_H(y)/\e (y) $ for each $v_r \in \varphi_i ^{-1}[y]$. Since in forming $H_{i+1}$ no edge is detached from $v_r$ for each $v_r \in \varphi_i ^{-1}[y]$, we have $d_{H_{i+1}}(v_{r})=d_{H_i}(v_r)$. Therefore $d_{H_{i+1}}  (v_r) \approx d_H(y)/\e (y) $ for each $v_r \in \varphi_{i+1} ^{-1}[y]$. This proves (C3)$'$(ii). 
\item [\textup{(C4)$'$}] The proof is similar to the proof of  (C3)$'$, following from (B4) and (C4) of the induction hypothesis.   
\item [\textup{(C5)$'$}] By (B5)(i), $m_{H_{i+1}} (y, v_r)/ \e_ {i+1}(y) \approx  m_{H_i} (y, v_r)/\e _i(y)$ for each $v_r \in \varphi_i ^{-1}[y]$. By (C5)(i) of the induction hypothesis, $m_{H_i}   (y, v_r)/\e _i(y) \approx \ell_H(y)/\binom {\e (y)}{2}$ for each $v_r \in \varphi_i ^{-1}[y]$. Therefore
\begin{eqnarray*}
\frac{m_{H_{i+1}} (y, v_r)}{ \e_ {i+1}(y)}& \approx & \frac {\ell_H(y)}{\binom {\e (y)}{2}}.
 \end{eqnarray*}
for each $v_r \in \varphi_i ^{-1}[y]$. Moreover, by (B5)(iii) $m_{H_{i+1}}(y, v_{i+1})/\e_{i+1}(y) \approx  \ell_{H_i}(y)/\binom{\e _i(y)}{2}$, and by (C1)(i) of the induction hypothesis, $\ell _{H_i}(y)  \approx  \ell_H(y)\binom {\e _i(y)}{2}/\binom {\e(y)}{2}$. Therefore 
\begin{eqnarray*}
\frac{m_{H_{i+1}} (y, v_{i+1})}{ \e_ {i+1}(y)}& \approx &   \frac {\ell_H(y)\binom {\e _i(y)}{2}}{\binom {\e(y)}{2}\binom {\e _i(y)}{2}}= \frac {\ell_H(y)}{\binom {\e (y)}{2}}.
 \end{eqnarray*}
This proves (C5)$'$(i). 

By (B5)(ii), $m_{H_{i+1}} (v_{i+1}, v_r) \approx   m_{H_i} (y, v_r) /\e _i(y)$ for each $v_r \in \varphi_i ^{-1}[y]$. By (C5)(i) of the induction hypothesis, $m_{H_i}   (y, v_r)/\e _i(y) \approx \ell_H(y)/\binom {\e (y)}{2}$ for each $v_r \in \varphi_i ^{-1}[y]$. Therefore $$m_{H_{i+1}} (v_{i+1}, v_r) \approx\frac {\ell_H(y)}{\binom {\e (y)}{2}}$$ for each $v_r \in \varphi_i ^{-1}[y]$. By (C5)(ii) of the induction hypothesis, $m_{H_i} (v_r, v_s) \approx \ell_H(y)/\binom {\e (y)}{2}$ for every pair of distinct vertices $v_r,v_s \in \varphi_i ^{-1}[y]$. Since in forming $H_{i+1}$ no edge is detached from $v_r$ for each $v_r \in \varphi_i ^{-1}[y]$, we have $m_{H_{i+1}} (v_r, v_s) =m_{H_i} (v_r, v_s)$. Therefore $$m_{H_{i+1}} (v_r, v_s) \approx \frac {\ell_H(y)}{\binom {\e (y)}{2}}$$ for every pair of distinct vertices $v_r,v_s \in \varphi_{i+1} ^{-1}[y]$. This proves (C5)$'$(ii). 
\item [\textup{(C6)$'$}] The proof is similar to the proof of (C5)$'$, following from (B6) and (C6) of the induction hypothesis. 

\item [\textup{(C7)$'$}] If $z\notin N_H(w)$ then $m_H(w,z)=0$ and (C7)$'$ is trivial. So we assume that $z\in N_H(w)$. 

\textbf{(i)} If $\e_{i+1}(w)=\e_i(w)-1$ (so $w=y$), by (B5)(i) $m_{H_{i+1}} (y, z)/ \e_ {i+1}(y) \approx  m_{H_i} (y, z) /\e _i(y)$, and since $\e _{i+1}(z)=\e _{i}(z)$, we have  $m_{H_{i+1}} (y, z)/ (\e_ {i+1}(y)\e_ {i+1}(z)) \approx m_{H_i} (y, z) /(\e _i(y)\e_ {i}(z))$. By (C7)(i) of the induction hypothesis, $m_{H_i} (y, z)/ (\e _i(y) \e _i(z)) \approx  m_H(y, z)/(\e (y) \e(z))$. Therefore 
\begin{eqnarray*}
\frac{m_{H_{i+1}} (y, z)}{ \e _{i+1}(y) \e _{i+1}(z)} & \approx &   \frac {m_H(y, z)}{\e (y) \e(z)}.
 \end{eqnarray*}
The other case, $\e_{i+1}(z)=\e_i(z)-1$), is similar. This proves (C7)$'$(i).

\textbf{(ii)} By (C7)(ii) of the induction hypothesis $m_{H_i}(v_r, v_s) \approx m_H(w, z)/(\e (w) \e (z)) $ for each $v_r\in\varphi_{i}^{-1}[w]$  and each $v_s\in\varphi_{i}^{-1}[z]=\varphi_{i+1}^{-1}[z]$. Since in forming $H_{i+1}$ no edge is detached from $v_r$ and  $v_s$  for each $v_r\in\varphi_{i}^{-1}[w]$  and each $v_s\in\varphi_{i}^{-1}[z]$, we have $m_{H_{i+1}}(v_r, v_s)=m_{H_i}(v_r, v_s)$. Therefore $m_{H_{i+1}}(v_r, v_s) \approx m_H(w, z)/(\e (w) \e (z)) $ for each $v_r\in\varphi_{i}^{-1}[w]$  and each $v_s\in\varphi_{i+1}^{-1}[z]$. If $\e_{i+1}(y)=\e_i(y)-1$ (so $w=y$), by  (B5)(ii) $m_{H_{i+1}} (v_{i+1}, v_s) \approx   m_{H_i} (y, v_s) /\e _i(y)$ for each $v_s \in \varphi_i ^{-1}[z]= \varphi_{i+1} ^{-1}[z]$. By (C7)(iii) of induction hypothesis, $m_{H_i} (y, v_s)/\e _i(y)  \approx m_H(y,    z)/(\e (y) \e (z))$. So $$m_{H_{i+1}} (v_{i+1}, v_s) \approx \frac {m_H(y,    z)}{\e (y) \e (z)}.$$ The other case, $\e_{i+1}(z)=\e_i(z)-1$, is similar. This proves (C7)$'$(ii).

\textbf{(iii)} If $\e_{i+1}(y)=\e_i(y)-1$ (so $w=y$), then by  (B5)(i) $m_{H_{i+1}}(y, v_s)/ \e_ {i+1}(y) \approx  m_{H_i} (y, v_s) /\e _i(y)$ for each $v_s \in \varphi_i ^{-1}[z]= \varphi_{i+1} ^{-1}[z]$. But by (C7)(iii) of induction hypothesis, $m_{H_i} (y, v_s) /\e _i(y)\approx m_H(y,    z)/(\e (y) \e (z))$ for each $v_s \in \varphi_i ^{-1}[z]$. Therefore 
\begin{eqnarray*}
\frac{m_{H_{i+1}} (y, v_s)}{ \e_ {i+1}(y)} & \approx &   \frac {m_{H} (y, z) }{\e (y)\e(z)}
 \end{eqnarray*}
 for each $v_s \in \varphi_{i+1} ^{-1}[z]$. If $\e_{i+1}(z)=\e_i(z)-1$ (so $z=y$), then since in forming $H_{i+1}$ no edge is detached from $v_s$ for each $v_s\in\varphi_{i}^{-1}[y]$, we have $m_{H_{i+1}}(w, v_s)=m_{H_i}(w,v_s)$ for each $v_s\in\varphi_{i}^{-1}[y]$. Therefore $m_{H_{i+1}}(w, v_s)/\e_{i+1}(w)=m_{H_i}(w,v_s)/\e_i(w)$ for each $v_s\in\varphi_{i}^{-1}[y]$. Moreover, by (B5)(ii) $m_{H_{i+1}}(w,v_{i+1})\approx m_{H_i}(w,y)/\e_i(y)$. Therefore  $m_{H_{i+1}}(w,v_{i+1})/\e_{i+1}(w)\approx m_{H_i}(w,y)/(\e_i(w)\e_i(y))$. By (C7)(i) of induction hypothesis, $m_{H_i}(w,y)/(\e_i(w)\e_i(y))=m_{H}(w,y)/(\e(w)\e(y))$. Hence 
 $$\frac{m_{H_{i+1}}(w,v_{i+1})}{\e_{i+1}(w)}\approx \frac{m_{H}(w,y)}{\e(w)\e(y)}.$$ 
 This proves (C7)$'$(iii).
 \item [\textup{(C8)$'$}] The proof is similar to the proof of (C7)$'$, following from (B6) and (C8) of the induction hypothesis. 

\end{itemize}
As a result of (C1)-(C8), we prove that $G$ is loopless, and satisfies conditions \textup{(A1)-(A6)} of Theorem  \ref{mainth}. 
Recall that $H_n=G$, $\varphi _n=\psi$, and $\e_n(w)=1$ for each $w\in V$. 
Let $w,z$ be an arbitrary pair of distinct vertices of $V$, and let $j\in \mathbb{Z}_k$.  
Now in (C1)-(C8) we let $i=n$. From C1(ii) it is immediate that $G$ is loopless. 

From (C3)(i) it follows that $d_{H_n}(w)/ \e _n(w) \approx d_H(w)/\e (w)$, so  $d_{G}(w) \approx d_H(w)/\e (w)$. From (C3)(ii), $d_{H_n}  (v_r) \approx d_H(w)/\e (w) $ for each $v_r\in\varphi _n^{-1}[w]$, so $d_{G}  (v_r) \approx d_H(w)/\e (w) $ for each $v_r\in\psi ^{-1}[w]$. Therefore $G$ satisfies (A1). 

From (C5)(i) it follows that  $m_{H_n}(w, v_r)/\e _n(w) \approx \ell_H(w)/\binom {\e (w)}{2}$ for each $v_r\in\varphi _n^{-1}[w]$, so $m_{G}(w, v_r) \approx \ell_H(w)/\binom {\e (w)}{2}$ for each $v_r\in\psi ^{-1}[w]$. From (C5)(ii), $m_{H_n} (v_r, v_s) \approx \ell_H(w)/\binom {\e (w)}{2}$ for every pair of distinct vertices $v_r,v_s \in \varphi_n^{-1}[w]$, so $m_{G} (v_r, v_s) \approx \ell_H(w)/\binom {\e (w)}{2}$ for every pair of distinct vertices $v_r,v_s \in \psi^{-1}[w]$. Therefore $G$ satisfies (A3). 

From (C7)(i) it follows that $m_{H_n} (w, z)/ (\e _n(w) \e _n(z)) \approx  m_H(w, z)/(\e (w) \e(z))$, so $m_{G} (w, z) \approx  m_H(w, z)/(\e (w) \e(z))$. From (C7)(ii), $m_{H_n}(v_r, v_s) \approx m_H(w, z)/(\e (w) \e (z)) $ for each $v_r\in\varphi_{n}^{-1}[w]$  and each $v_s\in\varphi_{n}^{-1}[z]$, so $m_{G}(v_r, v_s) \approx m_H(w, z)/(\e (w) \e (z)) $ for each $v_r\in\psi^{-1}[w]$  and each $v_s\in\psi^{-1}[z]$. From (C7)(iii) it follows that $ m_{H_n} (v _r, z)/\e _n(z) \approx m_H(w,    z)/(\e (w) \e (z)) $ for each $v_r\in \varphi^{-1}_n[w]$, so $ m_{G} (v _r, z)\approx m_H(w,    z)/(\e (w) \e (z)) $ for each $v_r\in \psi^{-1}[w]$. From (C7)(iii), $m_{H_n} (w, v_s)/\e _m(w)  \approx m_H(w,    z)/(\e (w) \e (z))$ for each $v_s\in \varphi^{-1}_n[z]$, so $m_{G} (w, v_s) \approx m_H(w,    z)/(\e (w) \e (z))$ for each $v_s\in \psi^{-1}[z]$. Therefore $G$ satisfies (A5). 

A similar argument shows that $G$ satisfies (A2), (A4), (A6). In order to prove that $G$ satisfies the last condition (A7) of Theorem \ref{mainth}, it suffices to show that if  for some $j\in\mathbb{Z}_k$, $d_{H_i(j)} (v)/ \e _i(v)$ is even for all $v \in V_i$, then 
 \begin{itemize}
 \item [\textup{(D1)}]  $d_{H_{i+1}(j)} (v)/\e _{i+1}(v)$ is an even integer for all $v \in V_{i+1},$ and  
 \item  [\textup{(D2)}] $\omega (H_{i+1} (j)) =  \omega(H_i(j))$.
 \end{itemize}
For then, if for each $v \in V(H)=V_0$, $d_{H(j)} (v)/ \e (v)=d_{H_0(j)} (v)/ \e _0(v)$ is an even integer, then it follows inductively that for each $0\leq r\leq n$ and each $v \in V_r$, $d_{H_r(j)} (v)/ \e _r(v)$ is an even integer and 
$$\omega (H_r(j))=\omega (H_0 (j)).$$
Therefore $\omega (G(j))=\omega (H_n(j))=\omega (H_0 (j))=\omega (H (j)).$ This will complete the proof of Theorem \ref{mainth}.

So we now establish (D1) and  (D2). Let $j \in \mathbb{Z}_k$ be a color for which for all $v \in V_i$, $d_{H_i(j)} (v)/ \e _i(v)$ is an even integer. Recall that $y$ is the vertex for which $\e_{i+1}(y)=\e_i(y)-1$. To establish (D1), there are three cases to consider:
\begin{itemize}
\item Case 1: $v \notin \{y, v _{i+1}\}$. Clearly $d_{H_{i+1}(j)}(v) =d_{H_i(j)}(v) $ and $\e _{i+1} (v)= \e _{i} (v)$.  So $d_{H_{i+1}(j)}(v)/\e  _{i+1}(v) = d_{H_i(j)}(v)/\e _i(v)$ which is an even integer.
\item Case 2: $v =y$. From (B4)(i), it follows that $d_{H_{i+1}(j)}(y)/ \e_ {i+1}(y) =  d_{H_i(j)}(y)/\e _i(y)$ which is an even integer.
\item Case 3: $v = v _{i+1}$. From (B4)(ii), it follows that $d_{H_{i+1}(j)}(v_{i+1}) = d_{H_i(j)}(y)/\e_i(y)$ which is an even integer. 
\end{itemize}
This proves (D1).

In order to prove (D2), let $H^y_i(j)$ be the component of $H_i(j)$ which contains $y$. It is enough to show that $\omega(H^y_{i+1}(j))=\omega(H^y_i(j))$. Let $\omega_{ij}=\omega(H^y_i(j)\backslash\{y\})$ and let $\Gamma_{i,j, 1}, \dots, \Gamma_{i,j,\omega_{ij}}$ be the vertex sets of the components of $H^y_i(j)\backslash\{y\}$. 
Note that  $\Gamma_{i,j, r}$ is a subset of  $V(B_i)$, of $V(T_i)$, and of $V(B'_i)$ for  $1\leq r\leq \omega_{ij}$. 
\begin{figure}[htbp]
\begin{center}
\scalebox{.65}{ \includegraphics  {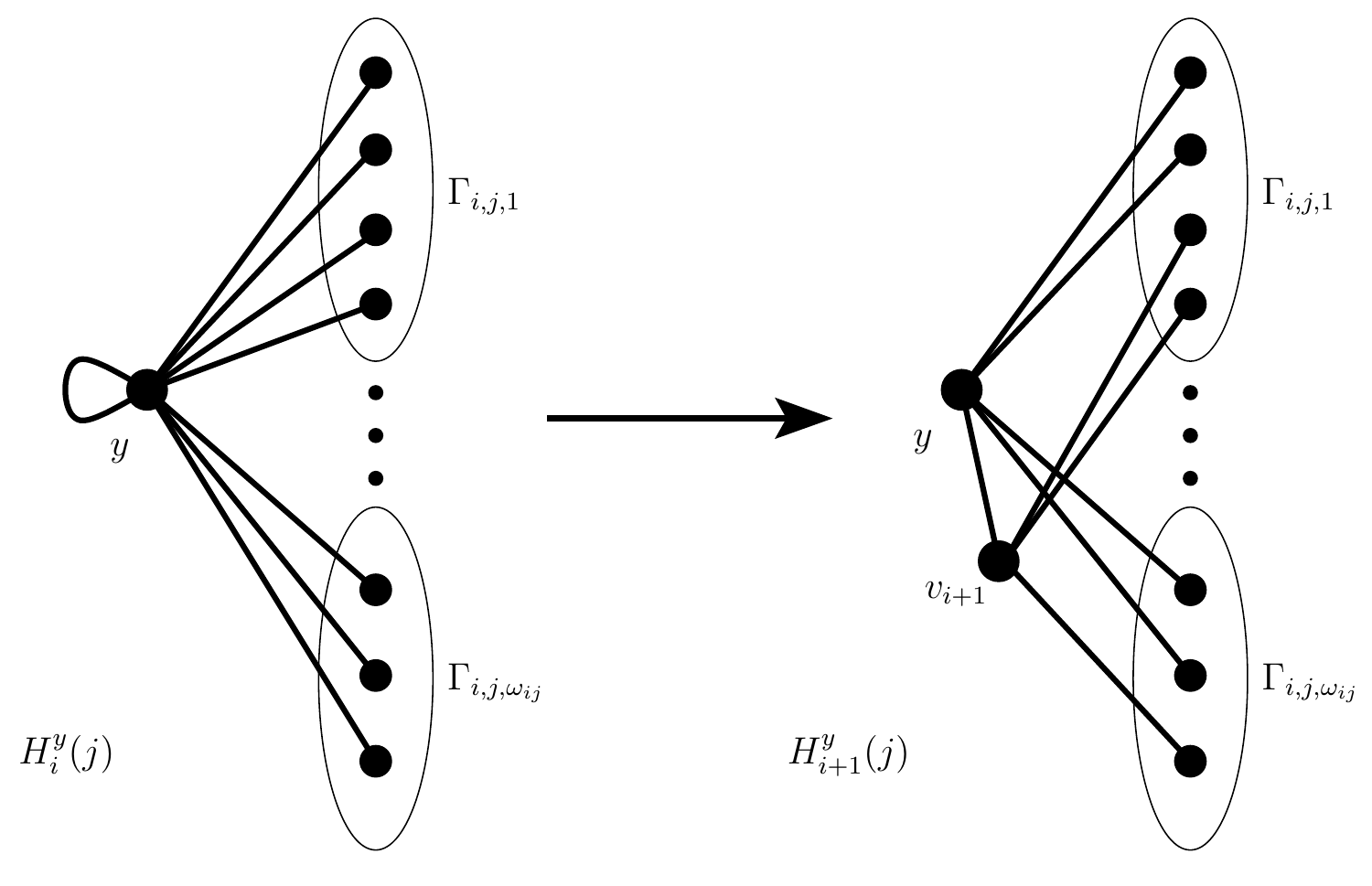} }
\caption{Detachment of $H_{i}^y(j)$ into $H_{i+1}^y(j)$}
\label{figure:H_i+1(j)y}
\end{center}
\end{figure}
Since $d_{H_i(j)}(v)/\e _i(v)$ is an even integer for each $v\in V_i$, it follows that $d_{H_i(j)}(v)$ is an even integer for each $v\in V_i$. Therefore $H_i(j)$ is an even graph (all vertices are of even degree). Since $d_{H_i(j)}(y)$ is even, so is $d_{H_i(j)}(y) - 2 \ell _{H_i(j)}(y)$. 
 Since $H_i(j)$ is an even graph, and the sum of the degree of the vertices in any graph must be even, it follows that $m_{H_i(j)}(y, \Gamma_{i,j, t})=m_{B_i}(c_{ij}, \Gamma_{i,j, t})$ is even for $1 \leq t \leq \omega_{ij}$. (In fact every edge cut in $H_i(j)$ is even.) 
Now from (M2) it follows that for each $t$, $1 \leq t \leq \omega_{ij}$, $m_{B_i'(1)}({\cal C}_{ij},  \Gamma_{i,j, t}) \approx m_{B'_i}({\cal C}_{ij},  \Gamma_{i,j, t})/2$. There are two cases to consider:
\begin{itemize}
\item Case 1: $m_{T_i}(c_{ij},  \Gamma_{i,j, t}) = m_{B_i}(c_{ij},  \Gamma_{i,j, t})$. In this case we have
$$m_{B_i'(1)}({\cal C}_{ij},  \Gamma_{i,j, t}) = \frac{m_{B'_i}({\cal C}_{ij},  \Gamma_{i,j, t})}{2}=\frac{m_{T_i}(c_{ij},  \Gamma_{i,j, t})}{2} = \frac{m_{B_i}(c_{ij},  \Gamma_{i,j, t})}{2}.$$
\item Case 2: $m_{T_i}(c_{ij},  \Gamma_{i,j, t}) < m_{B_i}(c_{ij},  \Gamma_{i,j, t})$. In this case we have
$$m_{B_i'(1)}({\cal C}_{ij},  \Gamma_{i,j, t}) \approx \frac{m_{B'_i}({\cal C}_{ij},  \Gamma_{i,j, t})}{2}=\frac{m_{T_i}(c_{ij},  \Gamma_{i,j, t})}{2} < \frac{m_{B_i}(c_{ij},  \Gamma_{i,j, t})}{2}.$$
\end{itemize}
Therefore in both cases $m_{B_i'(1)}({\cal C}_{ij},  \Gamma_{i,j, t})\leq m_{B_i}(c_{ij},  \Gamma_{i,j, t})/2$ for $1 \leq t \leq \omega_{ij}$. 
This is shown in Figure \ref{figure:H_i+1(j)y}.
This means, at most half of the edges joining $y$ to $ \Gamma_{i,j, t}$, $1\leq t \leq \omega_{ij}$, are moved to $v_{i+1}$ in forming $H_{i+1}$. So from each vertex $u\neq v_{i+1}$ in $H^y_{i+1}(j)$, there is a path of edges colored $j$ from $u$ to $y$. Moreover, $v_{i+1}$ is either adjacent with $y$ or is adjacent with another vertex in $H^y_{i+1}(j)$, so $v_{i+1}$ is also joined to $y$ by a path of edges colored $j$. Therefore $\omega(H^y_{i+1}(j))=\omega(H^y_i(j))$. This proves (D2) and the proof of Theorem \ref{mainth} is complete.
\end{proof}
\section{Hamiltonian Decomposition of \mathversion{bold}{$K(a_1, \dots, a_p; \la _1, \la _2)$}} \label{hdka1ap}
Let $a_1,\dots,a_p\in \mathbb{N}$, and  $\la_1,\la_2\in\mathbb{N}\cup\{0\}$. 
Let $G$ be the graph $K(a_1, \dots, a_p; \la _1, \la _2)$. Recall that $G$ is a graph with $p$ parts $V_1,\ldots,V_p$, with $|V_i|=a_i$ for $1\leq i \leq p$, $m_G(u,v)=\lambda_1$ for every pair of distinct vertices $u,v\in V_i$ for $1\leq i \leq p$, and $m_G(u,v)=\lambda_2$ for each $u\in V_i, v\in V_j$ for $1\leq i < j \leq p$. 

A graph $G$ is said to be even if all of its vertices have even degree. Let $k\in \mathbb{N}$. We say that $G$ has an evenly-equitable $k$-edge-coloring if $G$ has a $k$-edge-coloring for which,  for each $v\in V(G)$
\begin{itemize}
\item [(i)] $d_{G(i)}(v)$ is even for each $i\in \mathbb{Z}_k$, and 
\item [(ii)] $|d_{G(i)}(v)-d_{G(j)}(v)|\in\{0,2\}$ for each $i,j\in \mathbb{Z}_k$.  
\end{itemize}
We need the following theorem of Hilton \cite{H1}:
\begin{theorem}\textup{(Hilton \cite[Theorem 8]{H1})}\label{eveneq}
Each finite even graph has an evenly-equitable $k$-edge-coloring for each $k\in \mathbb{N}$.
\end{theorem} 

Walecki's construction for Hamiltonian decomposition of $K_n$ and $K_n-F$ where $F$ is a 1-factor \cite{L}, easily provides the following result:
\begin{theorem}\label{hdcm}
The graph $\lambda K_n$ is Hamiltonian decomposable if and only if $\lambda (n-1)$ is an even integer. 
\end{theorem}

Using these two results, together with Theorem \ref{mainth}, now we are able to find necessary and sufficient conditions for $K(a_1, \dots, a_p; \la _1, \la _2)$ to be Hamiltonian decomposable. 
Let us first look at some trivial cases:
\begin{itemize}
\item [(i)] If $p=1$, then $G=\la_1K_{a_1}$ which by Theorem \ref{hdcm}, is Hamiltonian decomposable if and only if $\la_1 (a_1-1)$ is even. 
\item [(ii)] If $p>1$,$\la_2=0$, then $G=\bigcup\limits_{i=1}^p{\la_1K_{a_i}}$. Clearly $G$ is disconnected and so is not Hamiltonian decomposable. 
\item [(iii)] If $a_i=1$ for $1\leq i\leq p$, then $G=\la_2 K_p$ which is Hamiltonian decomposable if and only if $\la_2 (p-1)$ is even. 
\item [(iv)] If $\la_1= \la_2$, then $G= \la_1 K_{a_1+\dots+a_p}$ which is Hamiltonian decomposable if and only if $\la_1 (\sum\limits_{i=1}^{p}{a_i}-1)$ is even. 
\end{itemize}
We exclude the above four cases from our theorem:
\begin{theorem}\label{hdka1apl1l2}
Let $p> 1$, $\la_1\geq 0$, and $\la_2\geq 1$, with $\lambda_1 \neq \lambda_2$ be integers. Let $a_1,\ldots,a_p$ be positive integers with $a_1\leq\ldots \leq a_p$, and $a_p\geq 2$. Let $G=K(a_1, \dots, a_p; \la _1, \la _2)$. Then $G$ is Hamiltonian decomposable if and only if the following conditions are satisfied:
\begin{itemize}
\item [\textup{(i)}] $a_i=a_j:=a$ for $1\leq i<j\leq p;$
\item [\textup{(ii)}]$\la_1(a-1)+\la_2a(p-1)$ is an even integer\textup{;} 
\item [\textup{(iii)}]$\la_1\leq \la_2a(p-1).$
\end{itemize}
\end{theorem}
\begin{proof}
Let $s=\sum\limits_{i=1}^{p}{a_i}$. To prove the necessity, suppose $G$ is Hamiltonian decomposable. For $v\in V_i$, $1\leq i\leq p$, we have $d_G(v)=\la_1(a_i-1)+\la_2(s-a_i)$. Since $G$ is Hamiltonian decomposable, it is regular. So we have 
$\la_1(a_i-1)+\la_2(s-a_i)=\la_1(a_j-1)+\la_2(s-a_j)$ for every pair $1\leq i <  j\leq p$. Equivalently $\la_1(a_i-a_j)=\la_2(a_i-a_j)$. So $(\la_1-\la_2)(a_i-a_j)=0$ and since $\la_1\neq \la_2$, we have $a_i=a_j:=a$ for every pair $1\leq i <  j\leq p$. So we can assume that $G=K(a^{(p)};\la_1,\la_2)$. Therefore $s=pa$ and $d_G(v)=\la_1(a-1)+\la_2(pa-a)=\la_1(a-1)+\la_2a(p-1)$. Now by the Hamiltonian decomposability of $G$, the degree of each vertex
$$\la_1(a-1)+\la_2a(p-1) \text {    is an even integer}.$$

By the preceding paragraph, the number of Hamiltonian cycles of $G$ is $\frac{1}{2}(\la_1(a-1)+\la_2a(p-1))$. Let us say that an edge is pure if both of its endpoints belong to the same part. 
Each Hamiltonian cycle passes through every vertex of every part exactly once. Hence each Hamiltonian cycle contains at most $(a-1)$ pure edges from each part. Since the total number of pure edges in each part is $\la_1\binom {a}{2}$, we have 
$$\la_1\binom {a}{2} \leq \frac{(a-1)}{2}(\la_1(a-1)+\la_2a(p-1)).$$
So, 
$$ \frac{\la_1a(a-1)}{2} \leq \frac{(a-1)}{2}(\la_1(a-1)+\la_2a(p-1)).$$
Since $a>1$, it implies that $\la_1a \leq \la_1(a-1)+\la_2a(p-1)$. Thus $\la_1\leq \la_2a(p-1)$. 
Therefore conditions (i)-(iii) are necessary. Note that the necessity of condition (iii) can also be seen as an edge-connectivity issue. Of course $G$ has edge-connectivity at most $\lambda_2 a^2(p-1)$, as deleting all the edges incident with vertices in a fixed part disconnects the graph. Since $G$ has a Hamiltonian decomposition, it clearly  has degree equal to its edge-connectivity. Therefore, the degree of $G$, namely $\lambda_1(a-1)+\lambda_2a(p-1)$, is at most $\lambda_2 a^2(p-1)$.

To prove the sufficiency, suppose conditions (i)-(iii) are satisfied and let $H$ be a graph with $|V(H)|=p, \ell_H(y)=\la_1\binom{a}{2}$ for every $y\in V(H)$, and $m_H(y,z)=\la_2a^2$ for every pair $y,z\in V(H)$ and let $\eta$ be a function from $V(H)$ into $\mathbb{N}$ with $\eta(y)=a$ for all $y\in V(H)$. We note that $H$ is $(\la_1a(a-1)+\la_2a^2(p-1))$-regular. It is easy to see that $H$ is an amalgamation of $G$. In what follows we shall find an appropriate edge-coloring for $H$ and then we shall apply Theorem \ref{mainth}, to show that $H$ has a $\e$-detachment $G$ in which every color class induces a Hamiltonian cycle. 

Let $H^*$ be the spanning subgraph of $H$ whose edges are the non-loop edges of $H$. It is easy to see that $H^*\cong \la_2a^2K_p$. We claim that $\la_2a(p-1)$ is even. To see this, suppose $\la_2a(p-1)$ is odd; then $a$ is odd and $\la_1(a-1)$ is even. But then  $\la_1(a-1)+\la_2a(p-1)$ is odd, contradicting condition (ii) of the theorem. Therefore $\la_2a^2(p-1)$ is even and thus by Theorem \ref{hdcm}, $H^*$ is Hamiltonian decomposable.

Since $\la_2a^2K_p$ is $\la_2a^2(p-1)$-regular, it is decomposable into $\la_2a^2(p-1)/2$ Hamiltonian cycles by Theorem \ref{hdcm}. Now define $k=\big(\la_1(a-1)+\la_2a(p-1)\big)/2$. From (ii), $k$ is an integer. Now since $a > 1$ and $\la_2a(p-1)\geq \la_1$, we have the following sequence of equivalences:
$$
(a-1)(\la_2a(p-1)-\la_1)\geq 0 \Leftrightarrow \la_2a(p-1)(a-1)-\la_1(a-1)\geq 0 \Leftrightarrow $$
$$\la_2a^2(p-1)-\la_1(a-1)-\la_2a(p-1) \geq 0 \Leftrightarrow \frac{\la_2a^2(p-1)}{2} \geq \frac{\la_1(a-1)+\la_2a(p-1)}{2}.
$$
Hence, the number of Hamiltonian cycles in $H^*$ is at least $k$. Now let $\mathcal{C}_1,\ldots, \mathcal{C}_k$ be $k$ arbitrary Hamilton cycles of a Hamiltonian decomposition of $H^*$. Let $\mathcal{K}^*$ be a (partial) $k$-edge-coloring of $H^*$ such that all edges of each cycle $\mathcal{C}_i$ are colored $i$, for each $i\in \mathbb{Z}_k$. Now let $H^{**}$ be the spanning subgraph of $H$ whose edges are all the edges of $H$ that are uncolored in $H^*$. Recall that $H$ is $2ak$-regular, so for each $v\in V(H^{**})$ we have $d_{H^{**}}(v)=2ak-2k=2(a-1)k$. Therefore $H^{**}$ is an even graph and so by Theorem \ref{eveneq} it has an evenly-equitable edge-coloring $\mathcal{K}^{**}$ with $k$ colors $1,\ldots,k$ (Note that we are using the same colors we used to color edges of $H^*$). Therefore for each $j$, $1\leq j\leq k$, and for each $ y\in V(H^{**})$, we have $d_{H^{**}(j)}(y)=2(a-1)k/k=2(a-1)$. Now we can define the edges coloring $\mathcal{K}:E(H)\rightarrow \mathbb{Z}_k$ for $H$ as below:
$$
\mathcal{K}(e) = \left \{ \begin{array}{ll}
\mathcal{K}^*(e) & \mbox { if } e\in E(H^*)\backslash E(H^{**}),\\
\mathcal{K}^{**}(e) & \mbox { if } e\in E(H^{**}).
\end{array} \right. 
$$
So for each $j\in \mathbb{Z}_k$, for each $y\in V(H)$, we have $d_{H(j)}(y)=2+2(a-1)=2a$. Note that since all edges of each Hamiltonian cycle $\mathcal{C}_j$ are colored $j$, $1\leq j \leq k$, each color class $H(j)$ is connected.

So we have a $k$-edge-colored graph $H$ for which, for each $y,z\in V(H), y\neq z$, and each $j\in \mathbb{Z}_k$, $\eta(y)=a\geq 2$, $\ell_H(y)=\la_1\binom{a}{2}$, $m_H(y,z)=\la_2a^2$, $d_H(y)=2ak$, $d_{H(j)}(y)=2a$, $\omega(H(j))=1$.

Now by Theorem \ref{mainth}  there exists a loopless $\e$-detachment $G^*$ of $H$ with amalgamation function $\psi:V(G^*)\rightarrow V(H)$, $\e$  being the number function associated with $\psi$, such that for each $y,z\in V(H), y\neq z$, and each $j\in \mathbb{Z}_k$ the following conditions are satisfied:

\begin{itemize}
  \item $m_{G^*}(u,u')=\la_1\binom{a}{2}/\binom{a}{2}=\la_1$ for every pair of distinct vertices $u,u'\in \psi^{-1}(y)$;
  \item $m_{G^*}(u,v)=\la_2a^2/(aa)=\la_2$ for each $u\in \psi^{-1}(y)$ and each $v\in \psi^{-1}(z)$;
  \item $d_{G^*(j)}(u)=2a/a=2$ for each $u\in \psi^{-1}(y)$;
  \item $\omega(G^*(j))=\omega(H(j))=1$, since $d_{H(j)}(y)/\eta(y)=2a/a=2$. 
\end{itemize} 
From the first two conditions it follows that $G^*\cong K(a^{(p)};\la_1,\la_2)=G$. The last two conditions tells us that each color class is $2$-regular and connected, respectively; that is each color class is a Hamiltonian cycle. So we obtained a Hamiltonian decomposition of $K(a^{(p)};\la_1,\la_2)$ and the proof is complete. 
\end{proof}
\begin{remark}\textup{
 We may prove the necessity of condition (iii) of Theorem \ref{hdka1apl1l2} by a different counting argument. Let us say an edge is mixed if its endpoints are from different parts of $G$. Each Hamiltonian cycle starts from a vertex of a part $V_i$ for some $1\leq i \leq p$ and it will pass through every part at least once and it will eventually come back to the initial vertex in $V_i$.  Hence each Hamiltonian cycle contains at least $p$ mixed edges. On the other hand, the total number of mixed edges is $\la_2a^2\binom {p}{2}$. Therefore, 
$$\la_2a^2\binom {p}{2} \geq p \frac{1}{2}(\la_1(a-1)+\la_2a(p-1)).$$
So, 
$$ \frac{\la_2a^2p(p-1)}{2} \geq \frac{p(\la_1(a-1)+\la_2a(p-1))}{2}.$$
It implies that, $\la_2a(p-1)(a-1)-\la_1(a-1)\geq 0 $, so $(a-1)(\la_2a(p-1)-\la_1)\geq 0$ and since $a> 1$, we have $\la_1\leq \la_2a(p-1).$ 
}\end{remark}
\begin{remark}\textup{
 Observe that the equality in condition (iii) of Theorem \ref{hdka1apl1l2}  holds if and only if for each Hamiltonian decomposition, each Hamiltonian cycle contains exactly $(a-1)$ pure edges from every part, and exactly $p$ mixed edges. 
}\end{remark}

\section{Acknowledgement}
We would like to thank the referees for their helpful comments which improved the first draft of this paper.

\end{document}